\documentclass{amsart}

\usepackage{amssymb,amsfonts}
\usepackage[all,arc]{xy}
\usepackage{enumerate}
\usepackage{mathrsfs}
\usepackage{verbatim}
\usepackage{microtype}
\usepackage{mathtools}

\newcounter{derpp}

\newtheorem{thm}{Theorem}[section]
\newtheorem{cor}[thm]{Corollary}
\newtheorem{prop}[thm]{Proposition}
\newtheorem{lem}[thm]{Lemma}
\newtheorem{conj}[thm]{Conjecture}

\newtheorem{mthm}[derpp]{Theorem}
\newtheorem{mcor}[derpp]{Corollary}

\newtheorem{defn}[thm]{Definition}

\DeclareMathOperator{\Spec}{Spec}
\DeclareMathOperator{\Proj}{Proj}

\DeclareMathOperator{\Tor}{Tor}
\DeclareMathOperator{\gr}{gr}
\DeclareMathOperator{\Supp}{Supp}

\DeclareMathOperator{\red}{red}

\newcommand{\mco}{\mathcal{O}}

\newcommand{\fm}{\mathfrak{m}}
\newcommand{\fn}{\mathfrak{n}}
\newcommand{\fp}{\mathfrak{p}}
\newcommand{\fq}{\mathfrak{q}}

\newcommand{\bg}{\mathbf{g}}
\newcommand{\ds}{\displaystyle}

\newcommand{\oth}{\widehat{\otimes}}

\newcommand{\opi}{\overline{\pi}}

\newcommand{\wTor}{\widehat{\Tor}}

\makeatletter
\let\c@equation\c@thm
\def\anchorlessfootnote{\gdef\@thefnmark{}\@footnotetext}
\makeatother
\numberwithin{equation}{section}

\title{Positivity of Intersection Multiplicity for Power Series over a Two-Dimensional Base}
\author{C. Skalit}
\address{Department of Mathematics \\
University of Illinois at Chicago \\
851 S. Morgan St. \\
Chicago, IL 60607 \\
cskalit@gmail.com
}

\begin{document}
\title{Positivity of Intersection Multiplicity Over a Two-Dimensional Base}
\begin{abstract}We prove the positivity of Serre's Intersection Multiplicity for regular local rings that are essentially smooth over a two-dimensional, regular base. Afterward, we apply this result to prove a transversality theorem for unramified regular local rings via a local analysis on the blowup.
\end{abstract}
\maketitle
\section{Introduction}\anchorlessfootnote{\textit{2010 Mathematics Subject Classification.} Primary: 13D22. Secondary: 14C17. \\ \indent Keywords: Intersection multiplicity, Power-series, Blowup, Tangent cone.}
Given a regular local ring $A$ and two finitely-generated $A$-modules $M$ and $N$ whose supports intersect in a single point, Serre \cite{Serre} defines the intersection multiplicity of $M$ and $N$ via the (necessarily finite) sum
\[\chi^A(M,N) := \sum_{i \geq 0}{(-1)^i \ell(\Tor_i^A(M,N))}. \]
This formula, which prescribes how to multiply properly intersecting algebraic cycles, has served as the focal point for one of the long-standing conjectures in the homological theory of commutative rings:
\begin{conj}[Serre's Conjecture]\label{scc} Let $A$ be a regular local ring, and suppose that $M$ and $N$ are finitely-generated $A$ modules such that $\ell(M \otimes_A N) < \infty$. Then
\begin{itemize}
\item[(a)] $\dim M + \dim N \leq \dim A$. \textit{(Decency of intersection)}
\item[(b)] $\chi^A(M,N) \geq 0$. \textit{(Non-negativity)}
\item[(c)] $\chi^A(M,N) = 0$ if $\dim M + \dim N < \dim A$. (Vanishing)
\item[(d)] $\chi^A(M,N) > 0$ if $\dim M + \dim N = \dim A$. (Positivity)
\end{itemize}
\end{conj}
While part (a) was settled by Serre [ibid] using only the Cohen Structure Theorem, the proofs of the other parts of the conjecture in full generality have borrowed techniques from far outside the realm of commutative algebra. Part (b), which is due to Gabber \cite{Hochster, Berthelot}, uses de Jong's method of producing a generically-finite resolution of singularities \cite{deJong} to reduce the non-negativity of $\chi^A$ to showing that a certain vector-bundle is globally-generated. The vanishing in part (c) has been deduced by Gillet-Soul\'{e} \cite{Gillet2,Gillet} via K-theoretic techniques while a completely different proof by Paul Roberts \cite{Roberts} uses his machinery of local Chern characters.

\subsection*{The Main Theorem}
The general case of (d), which we shall henceforth refer to as Serre's Positivity Conjecture, still remains open. While some low-dimensional cases have been resolved, the most general result in this direction is Serre's \cite{Serre}, which states that if (the completion of) $A$ is a power series ring over a field or complete discrete valuation ring, then positivity holds. The central aim of this paper is to extend this result to power series over a two-dimensional base:

\begin{mthm}Let $(R,\fm)$ be a complete, regular local ring of dimension two, and let $A = R[[X_1, \cdots, X_m]]$. If $M$ and $N$ are finitely-generated $A$-modules satisfying $\dim M + \dim N = \dim A$ and $\ell(M \otimes_A N) < \infty$, then $\chi^A(M,N) > 0$.
\end{mthm}

For $R$ any complete, regular local ring and $A = R[[X_1, \cdots, X_m]]$, there is a reduction to the diagonal spectral sequence
\[ E^2_{pq} = \Tor_p^{A \oth_R A}((A \oth_R A)/\Delta, \wTor_q^R(M,N)) \Rightarrow \Tor_{p+q}^A(M,N) \]
where $\Delta$ is the diagonal ideal given by the kernel of the multiplication map $A\oth_R A \twoheadrightarrow A$.
When Serre proves the positivity for the case of $R$ a DVR, he shows that one can assume that $M$ is $R$-flat, meaning that this spectral sequence degenerates to give isomorphisms
\[ \Tor_p^{A \oth_R A}((A \oth_R A)/\Delta, M \oth_R N) \cong \Tor_p^A(M,N), \]
thus permitting us to realize $\chi^A$ as $e_{\Delta}(M \oth_R N)$, the Hilbert-Samuel multiplicity of $M \oth_R N$ with respect to $\Delta$. 

In our situation, when $R$ is two-dimensional, we can no longer be assured that the spectral sequence degenerates; and as a result, $\chi^A(M,N)$ sees contributions from the Hilbert-Samuel multiplicities of each of the higher $\wTor^R_i(M,N)$ modules:
\[ \chi^A(M,N) = \sum_{i=0}^2{(-1)^i e_{\Delta}(\wTor_i^R(M,N))} \]
(where, by convention, $e_{\Delta}(\wTor_i^R(M,N)) = 0$ if $\dim(\wTor_i^R(M,N)) < n$). Reduction to the case where $\wTor_2^R(M,N) = 0$ is fairly routine (cf. Lemma \ref{depth_detection2} below); thus, the chief obstruction to the positivity of $\chi^A(M,N)$ is precisely the $\wTor^R_1(M,N)$ term. Much of the proof of Theorem A is devoted to showing that this module is sufficiently ``small'' in the appropriate sense. Somewhat surprisingly, Gabber's non-negativity theorem plays a crucial role in achieving this end (Corollary \ref{gabber_cor}). After recalling some generalities concerning completed $\Tor$ in Section \ref{generalities}, we prove Theorem A over the course of Sections \ref{ct2} and \ref{proofA}.

As an immediate application of Theorem A, we prove that Serre's Positivity Conjecture holds for the local rings of a smooth scheme over a two dimensional base:
\begin{mcor}Let $f:X \to Y$ be a smooth morphism of schemes with $Y$ regular of dimension at most $2$. Then every local ring $\mco_{X,x}$ satisfies Serre's Positivity Conjecture.
\end{mcor}

The proof proceeds by reducing to the case of a power series ring. However, given an essentially smooth map of local rings $\mco_{Y,f(x)} \to \mco_{X,x}$, one cannot immediately conclude that $\widehat{\mco}_{X,x}$ is a power series ring if the residue field extension $k(x)/k(f(x))$ is inseparable. Some care is required to modify $\mco_{Y,f(x)}$ and $\mco_{X,x}$ in such a way that introduces separability without losing too much information. We address this technique in Section \ref{smooth_cor}.

\subsection*{Detecting Transversality}
Our second application is best phrased in geometric language. Let $X = \Spec A$ where $A$ is a regular local ring, and suppose that $Y$ and $Z$ are closed, integral subschemes that meet in a single point and whose dimensions add up to $\dim X$. When $A$ contains a field, Serre's proof of positivity actually shows that $\chi^{\mco_{X}}(\mco_Y,\mco_Z) \geq e(\mco_Y)e(\mco_Z)$. (Here, $e(\mco_Y)$ and $e(\mco_Z)$ are just the Hilbert-Samuel multiplicities of $\mco_Y$ and $\mco_Z$ with respect to the maximal ideal.) In \cite{Skalit} we show that this bound also holds when $A$ is of mixed-characteristic and unramified. Since $e(\mco_Y) = 1$ precisely when $Y$ is regular (and likewise for $Z$), the inequality $\chi^{\mco_X}(\mco_Y,\mco_Z) \geq e(\mco_Y)e(\mco_Z)$ is saying that intersection multiplicity is sensitive to the singularities on $Y$ and $Z$.

With this lower bound established (at least in the unramified case), it is natural to ask under what circumstances the intersection multiplicity is minimized. If we carry over the intuition from B\'{e}zout-type theorems in $\mathbb{P}_{\mathbb{C}}^n$, we would expect $\chi^{\mco_X}(\mco_Y,\mco_Z) = e(\mco_Y)e(\mco_Z)$ precisely when $Y$ and $Z$ meet ``transversely'' --- that is, when their tangent cones intersect trivially.\footnote{We are \textbf{not} assuming that $Y$, $Z$, or $Y \cap Z$ are non-singular; for us, ``transverse'' is solely a condition on the dimension of the intersection of the tangent cones.}  Algebraically, this amounts to saying that $\gr \mco_Y \otimes_{\gr \mco_X} \gr \mco_Z$ is zero-dimensional. One direction of this implication, that a transverse intersection implies the equality, is a result of Tennison \cite{Tennison} and holds for any regular local ring $A$. In \cite{Skalit} we put forth the following conjecture which asserts that $\chi^{\mco_X}$ alone is enough to detect transversality:

\begin{conj}[Transversality Conjecture]\label{transversal} Let $A$ be a regular local ring, and let $X = \Spec A$ with $Y$ and $Z$ be as above. Then $\chi^{\mco_X}(\mco_Y,\mco_Z) \geq e(\mco_Y)e(\mco_Z)$ with equality if and only if the tangent cones of $Y$ and $Z$ intersect trivially.
\end{conj}

As in \cite{Skalit}, we shall restrict our attention to the case were $A$ is unramified (and hence the inequality $\chi^{\mco_X}(\mco_Y,\mco_Z) \geq e(\mco_Y)e(\mco_Z)$ is already known to hold). Our main technique for exploring this conjecture is to construct the blowup $\phi:\widetilde{X} \to X$ of $X$ along its closed point and then consider the strict transforms $\widetilde{Y}$ and $\widetilde{Z}$ on $\widetilde{X}$. Set-theoretically, $\widetilde{Y} \cap \widetilde{Z}$ agrees with the intersection of the projectivized tangent cones of $Y$ and $Z$, which can be measured quantitatively via the formula \cite[20.4.3]{Fulton}:
\[ \chi^{\mco_{X}}(\mco_Y,\mco_Z) = e(\mco_Y)e(\mco_Z) + \chi^{\mco_{\widetilde{X}}}(\mco_{\widetilde{Y}},\mco_{\widetilde{Z}}) \]
whose ``error term'' is given by
\[ \chi^{\mco_{\widetilde{X}}}(\mco_{\widetilde{Y}},\mco_{\widetilde{Z}}) := \sum_{i,j\geq 0}{(-1)^{i+j}\ell(H^i(\widetilde{X},\Tor_j^{\mco_{\widetilde{X}}}(\mco_{\widetilde{Y}},\mco_{\widetilde{Z}})))}. \]
Proving our conjecture therefore amounts to showing that $\chi^{\mco_{\widetilde{X}}}(\mco_{\widetilde{Y}},\mco_{\widetilde{Z}}) > 0$ precisely when $\widetilde{Y} \cap \widetilde{Z} \neq \emptyset$. While blowing up simplifies our problem conceptually, it introduces the algebraic obstacle of ramification. If $(A,\fn)$ is an unramified regular local ring with $p \in \fn-\fn^2$, then $\widetilde{X}$ the blowup along $\fn$ will be ramified along a certain two-dimensional subscheme (see Corollary \ref{ramification_locus}). Nonetheless, as we show in Proposition \ref{struct_blowup}, each of these local rings completes to a power series over a two-dimensional base, thereby allowing us to invoke Theorem A and conclude:

\begin{mcor}Let $(A,\fn)$ be an unramified regular local ring. Put $X = \Spec A$ and let $\widetilde{X}$ be the blowup along $\fn$. Then Serre's Positivity Conjecture holds at every local ring $\mco_{\widetilde{X},x}$.
\end{mcor}

By our previous work in \cite{Skalit}, if $A$ is an unramified regular local ring, Conjecture \ref{transversal} is known to hold in many circumstances (see Theorem \ref{old} below). When $\widetilde{Y}$ and $\widetilde{Z}$ meet properly (i.e. as a finite --- possibly empty --- set of points), we can calculate $\chi^{\mco_{\widetilde{X}}}(\mco_{\widetilde{Y}},\mco_{\widetilde{Z}})$ in terms of intersection multiplicities on the stalks of $\mco_{\widetilde{X}}$. Since Corollary \ref{pos_blowup} assures us that each of these multiplicities will be positive, we now have

\begin{mcor}Let $(A,\fn)$ be an unramified regular local ring. Let $X = \Spec A$ and let $Y$ and $Z$ be closed, integral subschemes such that $\ell(\mco_Y \otimes_{\mco_X} \mco_Z) < \infty$ and $\dim Y + \dim Z = \dim X$. Suppose that $\chi^{\mco_X}(\mco_Y,\mco_Z) = e(\mco_Y)e(\mco_Z)$. If $\dim(\gr \mco_Y \otimes_{\gr \mco_X} \gr \mco_Z) \leq 1$, then, in fact, $\dim(\gr \mco_Y \otimes_{\gr \mco_X} \gr \mco_Z) = 0$.
\end{mcor}

\section{Generalities for Complete Tensor Products}\label{generalities}
\setcounter{derpp}{0}
The basic reference for this section is \cite[V]{Serre}. Let $(R,\fm)$ be a complete regular local ring and let $A = R[[X_1, \cdots, X_m]]$ and $B = R[[Y_1, \cdots, Y_n]]$. For finitely-generated $A$ and $B$ modules $M$ and $N$, recall that we define the ``completed-$\Tor$ functors'' via the formula:
\[ \widehat{\Tor}_i^R(M,N) := \lim_{\stackrel{\longleftarrow}{p,q}} \Tor^R_i(M/\fm_A^pM, N/\fm_B^qN) \]
where $\fm_A$ and $\fm_B$ denote the maximal ideals of $A$ and $B$. When $M$ is fixed, the functors $\wTor_i^R(M,-)$ define a $\delta$-functor from the category of finitely-generated $B$-modules to the category of finitely-generated $A \oth_R B$-modules. Indeed, from the Artin-Rees lemma, it is easy to deduce that a short exact sequence $0 \to N' \to N \to N'' \to 0$ of finitely-generated $A$ modules gives rise to the expected long exact sequence in $\wTor$. Furthermore, the completed torsion functors exhibit the following depth-sensitivity property that we shall repeatedly exploit:

\begin{prop}\label{depth_detection}{\cite[V-9(g)]{Serre}} With $R$, $A$, $B$, and $M$ as above, suppose that $(x_1, \cdots x_r) \subset \fm \subset R$ is an $M$-regular sequence. Then $\wTor_j^R(M,E) = 0$ for all $j > \dim R - r$ and all finitely-generated $B$-modules $E$.
\end{prop}

As an $R$-module, $B$ admits a regular sequence of length equal to the dimension of $R$. Thus, by Proposition \ref{depth_detection}, we see that $\wTor_i^R(M,P)$ vanishes for $i > 0$ and $P$ any projective $B$-module of finite rank. At once we see that our $\delta$-functor is coeffaceable and hence universal (cf. \cite[\S 2.4]{Weibel}, \cite{Tohoku}):

\begin{prop}\label{universal}\cite[V-10]{Serre} Let $R$, $A$, $B$, and $M$ be as above. Then the $\wTor_i^R(M,-)$ are the left-derived functors of $M \oth_R -$. \end{prop}

\begin{lem}\label{recast} Let $R$, $A$, and $B$ be as above and put $C = A \oth_R B$. Then for finitely-generated $A$ and $B$ modules $M$ and $N$, there are canonical identifications for all $p$:
\[ \Tor_p^C(M \oth_R B, A \oth_R N) \cong \wTor_p^R(M,N). \]
\end{lem}
\begin{proof}Let $C = A \oth_R B$. Consider the functors
\[	\begin{array}{rrcll}
	F: & \operatorname{f.g.} B-\operatorname{mod} & \longrightarrow & \operatorname{f.g.} C-\operatorname{mod} & N \mapsto A \oth_R N \\
	G: & \operatorname{f.g.} C-\operatorname{mod} & \longrightarrow & \operatorname{f.g.} C-\operatorname{mod} & E \mapsto (M \oth_R B) \otimes_C E \end{array} \]
It is clear that $F$ maps projectives to projectives and that $G(F(N))$ is naturally isomorphic to $M \oth_R N$. Invoking Proposition \ref{universal} and the spectral sequence for the composition of functors, we have
\[ E^2_{pq} = L_pG(L_qF(N)) \Rightarrow L_{p+q}(GF)(N) = \wTor^R_{p+q}(M,N). \]
On the other hand, $F$ is exact, so the spectral sequence degenerates to give isomorphisms
\[ \Tor_p^C(M \oth_R B, A \oth_R N) \cong L_pG(F(N)) \cong \wTor^R_p(M,N). \]
  \end{proof}

\begin{cor}\label{ct_support}With the notations of Lemma \ref{recast}, there are inclusions
\[ \Supp(\wTor^R_i(M,N)) \subset \Supp(M \oth_R N) \]
where the supports are regarded as subsets of $\Spec C = \Spec A \oth_R B$.
\end{cor}
\begin{proof}This follows from Lemma \ref{recast} and the basic fact that for any two finitely-generated $C$-modules $E_1$ and $E_2$, the $\Tor_i^C(E_1,E_2)$ vanish outside of $\Supp(E_1 \otimes_C E_2)$.
  \end{proof}

We shall require finer information than merely knowing the supports of the $\wTor_i^R(M,N)$. The next result, which follows from Gabber's celebrated non-negativity theorem \cite{Hochster}, allows us to extract cycle-theoretic information about the $\wTor_i^R(M,N)$ in terms of the multiplicities with which their minimal components appear.

\begin{cor}\label{gabber_cor}Suppose that $(R,\fm)$ be a complete, regular local ring. Let $A = R[[X_1, \cdots, X_m]]$ and $B = R[[Y_1, \cdots, Y_n]]$. Suppose $M$ and $N$ are finitely-generated $A$ and $B$ modules respectively. Then for every prime $\fp \subset A \oth_R B$ minimal in the support of $M \oth_R N$, there is an inequality
\[ \sum_{i=0}^{\infty}{(-1)^i \ell(\wTor_i^R(M,N))_{\fp}} \geq 0 \]
\end{cor}
\begin{proof}From Corollary \ref{ct_support}, we know that for any minimal $\fp \in \Supp(M \oth_R N)$, $\wTor_i^R(M,N)_{\fp}$ will be a finite-length $(A\oth_R B)_{\fp}$-module. Thus, the alternating sum defined above is a well-defined number. Now, by Lemma \ref{recast}, we have, for $i \geq 0$, isomorphisms
\[ \Tor_i^C(M \oth_R B, A \oth_R N) \cong \wTor_i^R(M,N) \]
where $C = A \oth_R B$. The advantage this perspective is that localizing the left-hand side at $\fp$ is well-understood:
\[ \Tor_i^{C_{\fp}}((M \oth_R B)_{\fp}, (A \oth_R N)_{\fp}) \cong \wTor_i^R(M,N)_{\fp}. \]
Since the right-hand side is a finite-length module, the supports of the $C_{\fp}$ modules $(M \oth_R B)_{\fp}$ and $(A \oth_R N)_{\fp}$ meet at a single point, so by Gabber's non-negativity of intersection multiplicity, we have
\[ \sum_{i=0}^{\infty}{(-1)^{i}\ell(\wTor_i^R(M,N)_{\fp})} = \chi^{C_{\fp}}((M \oth_R B)_{\fp},(A \oth_R N)_{\fp}) \geq 0. \]
  \end{proof}

The last ingredient we shall need is the so-called ``reduction to the diagonal'' spectral sequence. It appears in \cite[V-12]{Serre} for the case when $R$ is $1$-dimensional, but its derivation at no point uses any dimensionality assumption.
\begin{prop}\label{rtd_ss}Let $R$ be as above and suppose that $A = R[[X_1, \cdots, X_m]]$. Let $A \oth_R A \to A: x \oth y \mapsto xy$ be the canonical surjection, and let $\Delta$ be its kernel. Then for any pair of finitely-generated $A$-modules $M$ and $N$, there is a convergent spectral sequence
\[ E^2_{pq} = \Tor^{A \oth_R A}_p((A \oth_R A)/\Delta, \wTor^R_q(M,N)) \Rightarrow \Tor^A_{p+q}(M,N). \]
\end{prop}

\section{The Two-Dimensional Case}\label{ct2}

\subsection{Geometry of Completed Tensor Products}
We begin with a lemma which is a mild generalization of a well-known flatness criterion for Cohen-Macaulay schemes over a regular base (cf. \cite[23.1]{Matsumura}):
\begin{lem}\label{flat_crit}Let $R \to \mco$ be a local homomorphism of Noetherian local rings where $(R,\fm)$ is regular of dimension $2$ and $\mco$ is normal. Then $\mco$ is $R$-flat if and only if $\dim(\mco/\fm \mco) = \dim \mco - 2$.
\end{lem}
\begin{proof}If $\mco$ is $R$-flat, then the desired equality follows from the classical formula for fibre-dimension of a flat morphism \cite[15.1]{Matsumura}. For the converse, let us suppose that $\fm = (s,t)$. Since $R/(s,t)$ is a field, it will suffice, by invoking the local flatness criterion \cite[III.10.3.A]{Hartshorne}, to show that $(s,t)$ is an $\mco$-regular sequence. Recall that a ring $B$ satisfies Serre Condition $\mathbf{S_n}$ if for all $\fp \in \Spec B$, $\operatorname{depth} B_{\fp} \geq \min\left\{ n, \dim B_{\fp} \right\}$. It follows that for $n > 0$, any ring satisfying $\mathbf{S_n}$ can have no embedded primes. Now, $\mco$ is normal by assumption and therefore satisfies $\mathbf{S_2}$. As $\mco/(s,t)\mco$ has dimension $\dim \mco -2$, we conclude that $s$ lies outside every minimal prime of $\mco$ and hence, by the $\mathbf{S_2}$ condition, outside every associated prime. Thus, $s$ is a non-zerodivisor on $\mco$, and it is easily seen that $\mco/s\mco$ satisfies $\mathbf{S_1}$. It follows by the same reasoning that $t$ is $\mco/s\mco$-regular, completing the proof.
  \end{proof}

\begin{prop}\label{dimbound}Let $(R,\fm)$ be a complete $2$-dimensional regular local ring and suppose that $A=R[[X_1, \cdots, X_m]]$ and $B = R[[Y_1, \cdots, Y_n]]$. Let $\mco_1$ and $\mco_2$ be finite $A$ and $B$-algebras which are domains. Then $\dim(\mco_1 \oth_R \mco_2) \geq \dim(\mco_1) + \dim(\mco_2) - 2$.
\end{prop}
\begin{proof}First, observe that since each $\mco_i$ is complete, it then follows that the normalization $\widetilde{\mco}_i$ is finite over $\mco_i$. From this, we obtain the isomorphism $\widetilde{\mco}_1 \oth_R \mco_2 \cong \widetilde{\mco}_1 \otimes_{\mco_1} (\mco_1 \oth_R \mco_2)$ and the corresponding fibre diagram:
\[\xymatrix{
\Spec \widetilde{\mco}_1 \oth_R \mco_2 \ar[r] \ar[d] & \Spec \widetilde{\mco}_1 \ar[d] \\
\Spec \mco_1 \oth_R \mco_2 \ar[r] & \Spec \mco_1
}\]
Since the normalization $\Spec \widetilde{\mco}_1 \to \Spec \mco_1$ is finite and surjective, so too is $\Spec \widetilde{\mco}_1 \oth_R \mco_2 \to \Spec \mco_1 \oth_R \mco_2$. Similarly $\Spec \widetilde{\mco}_1 \oth_R \widetilde{\mco}_2 \to \Spec \widetilde{\mco}_1 \oth_R \mco_2$ is also finite and surjective whence we obtain the equalities
\[\dim \mco_1 \oth_R \mco_2 = \dim \widetilde{\mco}_1 \oth_R \mco_2 = \dim \widetilde{\mco}_1 \oth_R \widetilde{\mco}_2. \]
We are therefore free to assume that both $\mco_1$ and $\mco_2$ are normal.\\
\newline
\textbf{Case 1:} (Assume $\mco_1$ is $R$-flat.) Let $\fm = (s,t)$ and let $\overline{R}$ be the DVR $R/sR$. We then have an isomorphism:
\[ \frac{\mco_1 \oth_R \mco_2}{s(\mco_1 \oth_R \mco_2)} \cong \left( \frac{\mco_1}{s \mco_1} \right) \oth_{\overline{R}} \left( \frac{\mco_2}{s \mco_2} \right). \]
Since $\mco_1/s\mco_1$ is flat over $\overline{R}$, it follows from \cite[2.1(b)]{Skalit} that
\[ \dim \left( \frac{\mco_1}{s \mco_1} \right) \oth_{\overline{R}} \left( \frac{\mco_2}{s \mco_2} \right) = \dim \left( \frac{\mco_1}{s \mco_1} \right) + \dim \left( \frac{\mco_2}{s\mco_2} \right) - 1 = \dim \mco_1 + \dim \left( \frac{\mco_2}{s \mco_2} \right) - 2. \]
By the flatness of $\mco_1$ over $R$, we see that $s$ either annihilates $\mco_1 \oth_R \mco_2$ or acts as a non-zerodivisor, depending on how $s$ acts on the domain $\mco_2$. In either case, we get an equality
\[ \dim(\mco_1 \oth_R \mco_2) = \dim \mco_1 + \dim \mco_2 - 2.\]

\textbf{Case 2:} (Assume that neither $\mco_1$ nor $\mco_2$ is $R$-flat.) Since we are assuming that the $\mco_i$ are normal, we have that $\dim(\mco_i/\fm \mco_i) \geq \dim \mco_i - 1$ for $i=1,2$ by Lemma \ref{flat_crit}. Putting $k = R/\fm$, we have the isomorphism
\[ \frac{\mco_1 \oth_R \mco_2}{\fm(\mco_1 \oth_R \mco_2)} \cong \frac{\mco_1}{\fm \mco_1} \oth_k \frac{\mco_2}{\fm \mco_2} \]
and therefore
\[ \begin{array}{rclcl}\ds\dim  (\mco_1 \oth_R \mco_2) &\geq& \ds \dim \left( \frac{\mco_1 \oth_R \mco_2}{\fm(\mco_1 \oth_R \mco_2)} \right) &  & \\
\ds &=& \ds \dim \left( \frac{\mco_1}{\fm \mco_1} \right) + \dim \left( \frac{\mco_2}{\fm \mco_2} \right) &\geq& \dim \mco_1 + \dim \mco_2 - 2.\end{array} \]
  \end{proof}

\subsection{Some Additional Vanishing Results}
The following vanishing result for completed $\Tor$ uses a technique which is similar in spirit to Proposition 3.3 in \cite{HochsterCBMS}.
\begin{lem}\label{depth_detection2}Let $(R,\fm)$ be a complete $2$-dimensional regular local ring. Suppose that $A = R[[X_1, \cdots X_m]]$ and $B = R[[Y_1, \cdots Y_n]]$. Suppose that $\mco_1$ and $\mco_2$ are domains, finite over $A$ and $B$ respectively. Let $J_i = \ker(R \to \mco_i)$.
\begin{enumerate}
\item[(a)] If $J_1 \neq \fm$, then $\wTor_2^R(\mco_1,-) = 0$. Similarly, $\wTor^R_2(-,\mco_2) = 0$ if $J_2 \neq \fm$.
\item[(b)] If $J_1,J_2 \neq \fm$ and $J_1 \neq J_2$, then $\wTor^R_1(\mco_1,\mco_2) = 0$.
\end{enumerate}
\end{lem}
\begin{proof}(a) If $J_1 \neq \fm$ and $\mco_1$ is a domain, we can choose some $g \in \fm$ that is a non-zerodivisor on $\mco_1$. The result now follows from Proposition \ref{depth_detection}.

\noindent(b) Without loss of generality, assume $J_1 \not\subset J_2$. Choose some nonzero $f \in J_1-J_2$ and consider the exact sequence
\[ 0 \to \mco_2 \stackrel{f}{\rightarrow} \mco_2 \to \overline{\mco_2} \to 0 \]
where $\overline{\mco_2} = \mco_2/(f \mco_2)$. Applying $\mco_1 \oth_R -$ gives a long exact sequence in $\wTor$ (Proposition \ref{universal}):
\[ \cdots \to \wTor^R_2(\mco_1,\overline{\mco_2}) \to \wTor_1^R(\mco_1,\mco_2) \stackrel{f}{\rightarrow} \wTor_1^R(\mco_1,\mco_2) \to \cdots \]
From part (a), $\wTor_2^R(\mco_1,\overline{\mco_2}) = 0$. Hence, $f$ is a non-zerodivisor on $\wTor^R_1(\mco_1,\mco_2)$. On the other hand, $f \in J_1$ must annihilate $\wTor_1^R(\mco_1,\mco_2)$, thereby forcing this module to be $0$.
  \end{proof}

\begin{lem}\label{shrink}Let $(R,\fm)$ be a complete $2$-dimensional regular local ring. Suppose that $A = R[[X_1, \cdots X_m]]$ and $B = R[[Y_1, \cdots, Y_n]]$. Suppose that $\mco_1$ and $\mco_2$ are domains, finite over $A$ and $B$ respectively. Let $J_i = \ker(R \to \mco_i)$, and suppose that $J_1 = J_2 = (0)$. Let $\phi: R \to A \oth_R B$ be the canonical map and let $\fp \in \Supp(\mco_1 \oth_R \mco_2)$ be a minimal prime. If $\phi^{-1}(\fp) \neq (0)$, then $\fp \notin \Supp(\wTor_1^R(\mco_1,\mco_2))$.
\end{lem}
\begin{proof}Let $0 \neq f \in \phi^{-1}(\fp)$. As in the proof of Lemma \ref{depth_detection2}, form the exact sequence
\[ 0 \to \mco_2 \stackrel{f}{\rightarrow} \mco_2 \to \overline{\mco_2} \to 0 \]
and consider the long exact sequence
\[ \cdots \to \wTor^R_2(\mco_1,\overline{\mco_2}) \to \wTor_1^R(\mco_1,\mco_2) \stackrel{f}{\rightarrow} \wTor_1^R(\mco_1,\mco_2) \to \cdots \]
Once again, $\wTor^R_2(\mco_1, \overline{\mco_2}) = 0$ thanks to Lemma \ref{depth_detection2}(a); so $f$ is a non-zerodivisor on $\wTor_1^R(\mco_1,\mco_2)$. Since $\Supp(\wTor_1^R(\mco_1,\mco_2)) \subset \Supp(\mco_1 \oth_R \mco_2)$ by Corollary \ref{ct_support}, we see that if $\fp \in \Supp(\wTor_1^R(\mco_1,\mco_2))$, then it would have to be minimal. Thus $\wTor_1^R(\mco_1,\mco_2)_\fp$ is Artinian (or zero). Since $f \in \phi^{-1}(\fp)$, $f$ is a non-unit in $(A \oth_R B)_\fp$ and therefore must act nilpotently on $\wTor_1^R(\mco_1,\mco_2)_{\fp}$. 
We conclude that $\wTor_1^R(\mco_1,\mco_2)_{\fp} = 0$.  \end{proof}

\section{Proof of the Main Theorem}\label{proofA}
\begin{lem}\label{decency}Let $(R,\fm)$ be a complete $2$-dimensional regular local ring. Suppose that $A = R[[X_1, \cdots X_m]]$. Suppose that $\mco_1$ and $\mco_2$ are domains, finite over $A$, such that $\dim \mco_1 + \dim \mco_2 = \dim A$ and $\ell(\mco_1 \otimes_A \mco_2) < \infty$. Put $J_i = \ker(R \to \mco_i)$. Then $J_1 = J_2$ if and only if $J_1 = J_2 = (0)$.
\end{lem}
\begin{proof}Suppose that $J_1 = J_2 \neq (0)$. Since the $J_i$ are prime ideals and $R$ is a UFD, we can find a prime element $0 \neq f \in J_1 = J_2$. Put $S = R/fR$ and consider the integral closure $\tilde{S}$ in its field of fractions. Since $S$ is complete, $S \to \tilde{S}$ is a finite morphism, thereby forcing $\tilde{S}$ to be a local domain. Since $S$ has dimension $1$, $\tilde{S}$ is a DVR.

Put $B = S[[X_1, \cdots X_m]]$. By assumption, each $\mco_i$ is a finite $B$-module. Let $B' = \tilde{S}[[X_1, \cdots, X_m]]$, which is finite over $B$; therefore, the extensions $\mco_i' := \mco_i \otimes_B B'$ are finite over $\mco_i$. Now, $\mco_1' \otimes_{B'} \mco_2'$ is Artinian by virtue of being finite over $\mco_1 \otimes_B \mco_2$. Since $B'$ is a regular local ring, we must have $\dim \mco_1' + \dim \mco_2' \leq \dim B' < \dim A$ by the decency condition (a) of Serre's Conjecture (see \cite[V-18, Thm. 3]{Serre}).

From the argument used in Proposition \ref{dimbound}, $\Spec \mco'_i \to \Spec \mco_i$ is finite-surjective because $\Spec B' \to \Spec B$ is. Hence, $\dim \mco'_i = \dim \mco_i$, thereby contradicting the fact that $\dim \mco_1 + \dim \mco_2 = \dim A$.
  \end{proof}

\begin{mthm}\label{main}Let $(R,\fm)$ be a complete regular local ring of dimension $2$. Let $A = R[[X_1, \cdots X_m]]$. Suppose that $M$ and $N$ are finitely-generated $A$-modules such that $\ell(M \otimes_A N) < \infty$ and $\dim M + \dim N = \dim A$. Then $\chi^A(M,N) > 0$.
\end{mthm}
\begin{proof}Recall that if $\fp \in \operatorname{Ass}(M)$, we have a short exact sequence
\[ 0 \to A/\fp \to M \to M'' \to 0 \]
Proceeding inductively and using the biadditivity of $\chi^A(-,-)$, we are reduced to proving the claim for $M = A/\fp$ and $N = A/\fq$ where $\fp$ and $\fq$ are primes. Next, we observe that $A/\fp$ is complete, making the map to the normalization $A/\fp \to \widetilde{A/\fp}$ finite. If we consider the exact sequence
\[0 \to A/\fp \to \widetilde{A/\fp} \to C \to 0, \]
we see that $\dim C < \dim A/\fp$ as $A/\fp \to \widetilde{A/\fp}$ is generically an isomorphism. Repeating the argument for $N=A/\fq$ and invoking the vanishing result (c) of Serre's Conjecture shows that $\chi^A(A/\fp,A/\fq) = \chi^A(\widetilde{A/\fp},\widetilde{A/\fq})$. It therefore suffices to prove the theorem for $M = \mco_1$ and $N = \mco_2$ where each $\mco_i$ has the structure of a \textbf{normal} integral domain. Of the two main cases to consider, let us first dispense with the trivial one.

\textbf{Case 1:} ($\fm \cdot \mco_1 = 0$). For this, we let $\pi \in \fm-\fm^2$ and put $\overline{A} = A/\pi A$. Note that $\pi$ cannot annihilate $\mco_2$; otherwise, both $\mco_1$ and $\mco_2$ would be supported on $\overline{A}$. Since part (a) of Conjecture \ref{scc} is known to hold in full generality, it would then follow that $\dim \mco_1 + \dim \mco_2 \leq \dim \overline{A} < \dim A$. Since $\pi$ is $\mco_2$-regular, the base change-spectral sequence
\[ E^2_{pq} = \Tor_p^{\overline{A}}(\mco_1, \Tor_q^A(\mco_2,\overline{A})) \Rightarrow \Tor_{p+q}^A(\mco_1,\mco_2) \]
degenerates to give isomorphisms
\[ \Tor^{\overline{A}}_p(\mco_1,\overline{\mco_2}) \cong \Tor_p^A(\mco_1,\mco_2) \]
and therefore an equality $\chi^A(\mco_1,\mco_2) = \chi^{\overline{A}}(\mco_1,\overline{\mco_2})$. Since $\overline{A}$ is a power series over the DVR $R/\pi R$, the right hand side is known to be positive from the work of Serre.

\textbf{Case 2:} ($\fm \cdot \mco_1 \neq 0$). Let $\Delta$ be the kernel of the canonical surjection $A \oth_R A \to A$. Clearly, $\Delta$ is generated by the $A \oth_R A$-regular sequence $\left\{X_i \oth 1 - 1 \oth X_i\right\}_{(1 \leq i \leq m)}$. We therefore obtain, for any $A \oth_R A$-module $L$, isomorphisms
\[ \Tor_i^{A \oth_R A}((A \oth_R A)/\Delta, L) \cong H_i(\Delta,L) \]
where the right-hand side is Koszul homology with respect to the sequence defining $\Delta$. Recall that if $\ell(L/\Delta L) < \infty$, then by a classical result of Serre \cite[IV-12]{Serre}, 
\[ \sum_{i=0}^{m}{(-1)^i \ell(\Tor_i^{A \oth_R A}((A \oth_R A)/\Delta, L))} = \sum_{i=0}^{m}{(-1)^i \ell(H_i(\Delta,L))} = e_{\Delta}(L,m) \]
where $e_{\Delta}(L,m)$ is $0$ when $\dim L < m$ and coincides with the Hilbert-Samuel multiplicity of $L$ with respect to $\Delta$ otherwise.

We now consider the reduction to the diagonal spectral sequence (Proposition \ref{rtd_ss}):
\[ E^2_{pq} = \Tor_p^{A \oth_R A}((A \oth_R A)/\Delta, \wTor_q^R(\mco_1,\mco_2)) \Rightarrow \Tor_{p+q}^A(\mco_1, \mco_2) \]
Note that $\Tor_0^{A \oth_R A}((A \oth_R A)/\Delta, \mco_1 \oth_R \mco_2) = \mco_1 \otimes_A \mco_2$ and thus is Artinian. Since we have by Corollary \ref{ct_support} that $\Supp(\wTor_i(\mco_1,\mco_2)) \subset \Supp(\mco_1 \oth_R \mco_2)$, we conclude that every term on the $E^2$ page of our spectral sequence has finite length. Furthermore, the $E^{\infty}$ page is just the associated graded of $\Tor_*^A(\mco_1,\mco_2)$ with respect to the filtration induced by the spectral sequence, so we can compute $\chi^A(\mco_1,\mco_2)$ by summing over the $E^{\infty}$ page:
\[ \begin{array}{rcl}
	\chi^A(\mco_1,\mco_2) & = & \ds\sum_{p,q \geq 0}{(-1)^{p+q} \ell(E^{\infty}_{pq})} \\
											  & = & \ds\sum_{p,q \geq 0}{(-1)^{p+q} \ell(E^2_{pq})} \\
												& = & \ds\sum_{q \geq 0}{(-1)^{q} \sum_{p \geq 0}{ (-1)^p \ell(\Tor_p^{A \oth_R A}((A \oth_R A)/\Delta, \wTor_q^R(\mco_1,\mco_2)))}} \\
												& = & \ds\sum_{q \geq 0}{(-1)^{q} \sum_{p \geq 0}{ (-1)^p \ell(H_p(\Delta,\wTor_q^R(\mco_1,\mco_2)))}} \\
												& = & \ds\sum_{q \geq 0}{(-1)^q e_{\Delta} (\wTor^R_q(\mco_1,\mco_2), m)} \end{array}. \]
												
Once again, let $J_i = \ker(R \to \mco_i)$. By hypothesis $J_1 \neq \fm$. We may assume that $J_2 \neq \fm$; otherwise, we are back in the situation of Case 1. By Lemma \ref{depth_detection2}(a), $\wTor^R_2(\mco_1, \mco_2) = 0$, so our formula for intersection multiplicity becomes
\begin{equation*}
\tag{*}\chi^A(\mco_1,\mco_2) = e_{\Delta}(\mco_1 \oth_R \mco_2, m) - e_{\Delta}(\wTor_1^R(\mco_1,\mco_2), m).
\end{equation*}
If $J_1 \neq J_2$, then by Lemma \ref{depth_detection2}(b), $\wTor_1^R(\mco_1,\mco_2) = 0$. Note that $\dim(\mco_1 \oth_R \mco_2) = m$: Lemma \ref{dimbound} shows the dimension is bounded below by $\dim \mco_1 + \dim \mco_2 - 2 = m$; the reverse inequality follows from the fact that $(\mco_1 \oth_R \mco_2)/\Delta$ is Artinian. In this case, it follows that $\chi^A(\mco_1,\mco_2) = e_{\Delta}(\mco_1 \oth_R \mco_2, m) > 0$. We therefore need only consider what can happen when $J_1 = J_2$, and by Lemma \ref{decency}, we are reduced to the case where $J_1 = J_2 = (0)$.

Since $\Supp(\wTor_1^R(\mco_1,\mco_2)) \subset \Supp(\mco_1 \oth_R \mco_2)$ by Corollary \ref{ct_support}, we can apply the associativity formula for Hilbert-Samuel multiplicity \cite[V-2]{Serre} to (*) and obtain the formula:
\[ \chi^A(\mco_1, \mco_2) = \sum_{\mathclap{\substack{\fp \in \Supp(\mco_1 \oth_R \mco_2) \\ \dim( (A\oth_R A)/\fp) = m}}}{\left[\ell((\mco_1 \oth_R \mco_2)_{\fp})-\ell(\wTor_1^R(\mco_1,\mco_2)_{\fp})\right]e_{\Delta}((A\oth_R A)/\fp, m)}.\]
From Lemma \ref{gabber_cor}, we know that for each $\fp$ satisfying $\dim((A \oth_R A)/\fp) = m$,
\[\ell((\mco_1 \oth_R \mco_2)_{\fp})-\ell(\wTor_1^R(\mco_1,\mco_2)_{\fp}) \geq 0.\] 
It therefore suffices to show that there exists at least one $m$-dimensional prime $\fq \in \Supp(\mco_1 \oth_R \mco_2)$ for which $\wTor_1^R(\mco_1,\mco_2)_{\fq} = 0$. If either $\mco_1$ or $\mco_2$ is $R$-flat, then, of course $\wTor_1^R(\mco_1,\mco_2)$ vanishes, and the positivity of $\chi^A(\mco_1,\mco_2)$ is obvious. If neither are flat, then by invoking the normality of the $\mco_i$, we must have that $\dim(\mco_i/\fm \mco_i) = \dim \mco_i - 1$ (Lemma \ref{flat_crit}). Hence,
\[\dim\left(\frac{\mco_1 \oth_R \mco_2}{\fm (\mco_1 \oth_R \mco_2)}\right) = \dim\left( \frac{\mco_1}{\fm \mco_1} \right) + \dim \left( \frac{\mco_2}{\fm \mco_2} \right) = \dim \mco_1 + \dim \mco_2 - 2 = m.\]
Since $\mco_1 \oth_R \mco_2$ is also $m$-dimensional, we conclude that $\fm(\mco_1 \oth_R \mco_2)$ lies within a minimal prime $\fq \in \Supp(\mco_1 \oth_R \mco_2)$ such that $\dim((A \oth_R A)/\fq) = m$. By Lemma \ref{shrink}, it follows that $\wTor_1^R(\mco_1,\mco_2)_{\fq} = 0$, and the proof is complete.
  \end{proof}

\section{Positivity of Intersections over a Two-Dimensional Base}\label{smooth_cor}
Suppose that $X \to Y$ is a smooth morphism of Noetherian schemes where $Y$ is regular of dimension two. We shall prove in this section for each $x \in X$, Serre's Positivity Conjecture holds for $\mco_{X,x}$. We note that the case of a one-dimensional base has been settled by using Fulton-MacPherson intersection theory \cite[Example 20.2.2]{Fulton}. We shall deduce the two-dimensional case from the corresponding result for power series rings in the previous section.

\begin{lem}\label{triviality}Let $(A,\fm)$ be a regular local ring. Suppose that $(B,\fn)$ is a regular local ring for which Serre's Positivity Conjecture is known to hold. If there is a flat, local morphism $A \to B$ with $\dim A = \dim B$, then the conjecture must also hold for $A$.
\end{lem}
\begin{proof}Note that $\dim(B/\fm B) = 0$ by hypothesis; put $r = \ell_{B}(B/\fm B)$. Given an Artinian $A$-module $E$, it follows from tensoring its composition series with $B$ that $\ell_B(E \otimes_A B) = r\ell_A(E)$. Hence, for any finitely-generated $A$-module $M$, the Hilbert functions $\ell_A(M/\fm^n M)$ and $\ell_B((M \otimes_A B)/\fm^n(M \otimes_A B)) = \ell_B(M/\fm^n M \otimes_A B)$ will differ by a factor of $r$. The corresponding Hilbert polynomials will have the same degree, thereby showing that $\dim M = \dim M \otimes_A B$.

Now let $M$ and $N$ be two finitely generated $A$-modules for which $\ell(M \otimes_A N) < \infty$ and $\dim M + \dim N = \dim A$. These two conditions will be preserved after base-extension to $B$, so $\chi^B(M \otimes_A B,N \otimes_A B) > 0$. On the other hand,
\[ \Tor_p^A(M,N) \otimes_A B \cong \Tor^B_p(M \otimes_A B, N \otimes_A B) \]
for all $p$, so we have $\chi^A(M,N) = (1/r)\chi^B(M \otimes_A B, N \otimes_A B) > 0$.
  \end{proof}

\subsection{Structure Theorems}
It is well-known that if $X$ is a smooth variety over a field $k$, then by the Cohen Structure Theorem, $\widehat{\mco}_{X,x}$ is a power series ring over the residue field $k(x)$. When the base is no longer a field, we have a similarly-flavored structure theorem, provided that we include an additional separability hypothesis.

\begin{prop}\label{extension}\cite[10.3.1]{EGAIII} Let $(R_0,\fm_0)$ be a Noetherian local ring. Let $L$ be an extension field of $R_0/\fm_0$. Then there is a local ring $(R,\fm)$, flat over $R_0$, such that $\fm_0 R = \fm$ and $R/\fm = L$.
\end{prop}

\begin{lem}\label{gen_cohen}Let $(R,\fm) \to (A,\fn)$ be a flat, local morphism of local rings and that the induced map on residue fields is separable.
\begin{itemize}
\item[(a)] The standard map $R \to \hat{A}$ factors through a complete local ring $(R',\fm')$, which is flat over $R$ and satisfies $R'/\fm R' = R'/\fm' = A/\fn$.
\item[(b)] If $R$, $A$, and $A/\fm A$ are regular, then $\hat{A} \cong R'[[X_1, \cdots, X_d]]$ and $d = \dim(A/\fm A)$.
\end{itemize}
\end{lem}
\begin{proof}By Proposition \ref{extension}, we can find a flat, local extension $R'$ of $R$ such that $R'/\fm R' = A/\fn$. Replace $R'$ by its completion if necessary. As $R \to R'$ is flat and $R/\fm \to R'/\fm R' = A/\fn$ is separable ($0$-smooth), $R \to R'$ is $\fm R'$-smooth by \cite[28.10]{Matsumura}.

We claim that $R \to \hat{A}$ factors through $R'$. $R' \to A/\fn$ is obtained by quotienting by the maximal ideal. If we inductively assume that a factorization $R \to R' \to A/\fn^j$ has been constructed, the definition of formal smoothness (cf. \cite[\S 28]{Matsumura} guarantees a lift:
\[
\xymatrix{
R' \ar[r] \ar@{.>}[rd] & A/\fn^j \\
R \ar[u] \ar[r] & A/{\fn}^{j+1} \ar[u]
}
\]
Since $\ds \hat{A} = \lim_{\stackrel{\longleftarrow}{j}}{A/\fn^j}$, we obtain a map $R' \to \hat{A}$ that extends $R \to \hat{A}$, and (a) is proved.

For (b), suppose that the maximal ideal of $A/\fm A$ is generated by the regular sequence $\overline{x_1}, \cdots, \overline{x_d}$. Let $x_i \in A$ be lifts of $\overline{x_i}$.  We now define $\phi: R'[[X_1, \cdots, X_d]] \to \hat{A}$ via $X_i \mapsto x_i$. $R'[[X_1, \cdots, X_d]]$ and $\hat{A}$ are regular local rings of the same dimension and have the same residue field. Since $\fn \hat{A} = \fm \hat{A} + (x_1, \cdots, x_d)$, $\phi$ is surjective, and its kernel must be trivial for dimensional reasons.
  \end{proof}

An essentially smooth morphism of local rings can very often fail to be separable at the level of residue fields. Take, for example, $\mathbb{F}_2(T) \to \mathbb{F}_2(T)[X]_{(X^2-T)}$. In these situations, we can first enlarge the residue field of the base to its perfect closure to ``force'' separability and put us into a situation where Lemma \ref{gen_cohen} may be applied.

\begin{lem}\label{tpfields}Let $k$ be a field and suppose that $E/k$ is a finitely-generated field extension, and $L/k$ is algebraic. Then $E \otimes_k L$ is an Artinian ring.
\end{lem}
\begin{proof}Since $k \to E$ is essentially of finite type, so too is $L \to E \otimes_k L$; hence $E \otimes_k L$ is Noetherian. Since $k \to L$ is algebraic, $E \to E \otimes_k L$ will be an integral extension, which forces $\dim E \otimes_k L = \dim E = 0$.
  \end{proof}

\begin{cor}\label{rel_cohen}Let $(R_0,\fm_0)$ be a regular local ring and suppose that $(R_0,\fm_0) \to (A_0, \fn_0)$ is an essentially smooth local morphism (i.e. the localization of a smooth, finitely-generated $R_0$-algebra). Then there exist flat extensions $(R,\fm)$ of $(R_0, \fm_0)$, $(A,\fn)$ of $(A_0, \fn_0)$, and an induced map $R \to A$ such that
\begin{enumerate}
\item $\fm_0 R = \fm$.
\item $A$ is essentially smooth over $R$.
\item $\fn_0 A$ is $\fn$-primary in $A$.
\item $\hat{A}$ is a power series ring over a complete regular ring $R'$ of dimension equal to $\dim R_0$.
\end{enumerate}
\end{cor}
\begin{proof}Let $L$ be the perfect closure of the field $k = R_0/\fm_0$. By Proposition \ref{extension}, there is a (necessarily regular) ring $(R,\fm)$ which is flat over $R$, has residue field $L$ and satisfies condition (1). The ring $A_0 \otimes_{R_0} R$ is essentially smooth over $R$ and hence regular. Since $\Spec(A_0 \otimes_{R_0} R) \to \Spec A_0$ is faithfully flat, we can find a prime ideal $\fp$ that maps to $\fn_0 \in \Spec A_0$. Setting $A = (A_0 \otimes_{R_0} R)_{\fp}$, we see that $A$ clearly satisfies (2). Next observe that by assumption, $E = A_0/\fn_0$ is a finitely-generated field extension of $k$. To prove (3), we note that $A/\fn_0 A$ is just the localization at $\fp$ of
\[  (A_0 \otimes_{R_0} R)/\fn_0(A_0 \otimes_{R_0} R) \cong (A_0 \otimes_{R_0} R) \otimes_{A_0} A_0/\fn_0 \cong E \otimes_{R_0} R \cong E \otimes_k L, \]
which by Lemma \ref{tpfields} is Artinian. Since $R$ has a perfect residue field, we can apply Lemma \ref{gen_cohen} and obtain (4).
  \end{proof}

\begin{mcor}Let $f:X \to Y$ be a smooth morphism of schemes with $Y$ regular of dimension at most $2$. Then every local ring $\mco_{X,x}$ satisfies Serre's Positivity Conjecture.
\end{mcor}
\begin{proof}If $x$ maps to $y \in Y$, then $\mco_{X,x}$ is essentially smooth over $\mco_{Y,y}$. By Corollary \ref{rel_cohen}, there is a flat local extension $(A,\fn)$ of $\mco_{X,x}$ inside of which the maximal ideal of $\mco_{X,x}$ is $\fn$-primary. Furthermore, $\hat{A}$ is a power series ring over a complete regular local ring $R'$ of dimension at most $2$. The claim now follows from Lemma \ref{triviality} and Theorem \ref{main}.
  \end{proof}

\section{Blowups and Transversality}
\subsection{The Locus of Ramification}For generalities concerning blowups, we refer the reader to \cite[Appx. B.6]{Fulton} or \cite[II.7]{Hartshorne}. In what follows, we shall confine ourselves to the following setting: Let $(A,\fn)$ be a regular local ring and let $X = \Spec A$. The blowup $\widetilde{X}$ of $X$ along $\fn$ may be realized via $\widetilde{X} = \Proj(A[\fn T])$ where $A[\fn T]$ is the Rees Algebra $\bigoplus_{i = 0}^{\infty}{\fn^i T^i}$. From the canonical morphism $\phi: \widetilde{X} \to X$, we have the exceptional divisor $E = \phi^{-1}(\Spec A/\fn) = \Proj(\gr_{\fn} A)$, which is just a projective space over the residue field having dimension one less than $X$. Fix $f \in \fn-\fn^2$ and consider $\overline{f} \in \fn/\fn^2$. We define the closed subscheme $V(\overline{f})$ of $E$ via the surjection $\ds \gr_{\fn} A \twoheadrightarrow \frac{\gr_{\fn} A}{\overline{f}\gr_{\fn} A}$. 

Each element $g \in \fn-\fn^2$ defines an affine patch $D_{+}(\bg)$ of $\widetilde{X}$ whose coordinate ring is $(A[\fn T])_{(\bg)}$, the degree-zero part of the $\mathbb{Z}$-graded ring $(A[\fn T])_{\bg}$ (cf. \cite[II.2]{Hartshorne}). Here, we write a bold $\bg$ for the degree-one element in $A[\fn T]$. If $\fn = (t_1, \cdots, t_d)$, then $D_{+}(\bg) = \Spec A[t_1/g, t_2/g, \cdots, t_d/g]$. We now give a local description of $V(\overline{f})$ inside of $\widetilde{X}$:

\begin{lem}\label{local_locus}With the above notations, $V(\overline{f}) \cap D_{+}(\bg)$ is defined by the ideal $(g,f/g)$ inside of $D_{+}(\bg)$.
\end{lem}
\begin{proof}On $D_{+}(\bg)$, it is clear from the construction that $g$ generates the exceptional divisor $E$. Since $V(\overline{f}) \hookrightarrow E$, it suffices to show the ideal of the closed immersion $V(\overline{f}) \cap D_{+}(\bg) \hookrightarrow E \cap D_{+}(\bg)$ is generated by (residue of) $f/g$. Consider the surjection of graded rings
\[ \gr_\fn A \twoheadrightarrow \frac{\gr_\fn A}{\overline{f} \gr_\fn A} \]
Let $\overline{g}$ be the degree-one element in $\fn/\fn^2$. Restricting to $E \cap D_{+}(\bg) = D_{+}(\overline{\bg})$ amounts to looking at the surjection
\[ (\gr_{\fn} A)_{(\overline{g})} \twoheadrightarrow \left( \frac{\gr_\fn A}{\overline{f} \gr_\fn A} \right)_{(\overline{g})} \]
It's clear that $\overline{f}/\overline{g}$ maps to $0$; we claim that it, in fact, generates the kernel. Suppose $\overline{h}/\overline{g}^k \in (\gr_{\fn} A)_{(\overline{g})}$ maps to $0$. This means that in $\gr_{\fn} A$, $\overline{f}$ divides $\overline{g}^m \overline{h}$ for some $m \in \mathbb{N}$. Since $\gr_{\fn} A$ is just a polynomial ring over a field, the element $\overline{f}$ is prime. If $\overline{f}$ divides $\overline{g}$, then $\overline{f}/\overline{g}$ is a unit, meaning that $D_{+}(\bg) \cap V(\overline{f}) = \emptyset$. Otherwise, $\overline{f}$ divides $\overline{h}$, meaning that $\overline{f}/\overline{g}$ divides $\overline{h}/\overline{g}$ as claimed.
  \end{proof}
From this lemma, we can compute the ``ramification locus'' of $f$ in $\widetilde{X}$, that is, the collection of points $x \in \widetilde{X}$ for which $f$ lies in the square of maximal ideal of $\mco_{\widetilde{X},x}$.

\begin{cor}\label{ramification_locus}For all $x \in \widetilde{X}$, $f$ lies in the square of the maximal ideal of $\mco_{\widetilde{X},x}$ if and only if $x \in V(\overline{f})$.
\end{cor}
\begin{proof}Of course, for any $\fp \in \Spec A$, we have $f \notin (\fp A_{\fp})^2$. Since $\widetilde{X} - E$ is isomorphic to the punctured spectrum $X - \left\{\fn\right\}$, we only need to consider those points $x \in E$. Choose some affine patch $D_{+}(\bg) \subset \widetilde{X}$ containing $x$. By Lemma \ref{local_locus}, $V(\overline{f}) \cap D_{+}(\bg)$ is cut out by the ideal $(g,f/g)$. If $x \in V(\overline{f})$, then both $g$ and $f/g$ lie in the the maximal ideal of $\mco_{\widetilde{X},x}$, meaning that $f = g \cdot f/g$ is in the square. Otherwise, if $x \in E - V(\overline{f})$, then only $g$, the local equation for $E$ lies in the maximal ideal of $\mco_{\widetilde{X},x}$ while $f/g$ is a unit. Hence, $\mco_{E,x} = \mco_{\widetilde{X},x}/(g) = \mco_{\widetilde{X},x}/(f)$ and the result follows from the fact that $\mco_{E,x}$ and $\mco_{\widetilde{X},x}$ are both regular.
  \end{proof}

We now specialize to the case of an unramified regular local ring. We introduce the following slightly generalized notion from \cite{Skalit}:
\begin{defn}Let $(R_0,\pi)$ be a DVR with perfect residue field. We call a regular local ring $(A,\fn)$ $R_0$-unramified if $R_0 \subset A$ and $\pi \in \fn-\fn^2$.
\end{defn}
When $R_0 = \mathbb{Z}_{p\mathbb{Z}}$, the notion of $R_0$-unramified coincides with the traditional definition of unramified. A regular local ring of mixed-characteristic can be ramified at $p \in \mathbb{Z}$ yet $R_0$-unramified for an appropriate choice of $R_0$. Such rings arise naturally when we take the ring of integers $\mco_K$ of a number field $K/\mathbb{Q}$ and consider the local rings of a smooth $\mco_K$-scheme $Y$. If $A$ contains a field $k$, and $\pi \in \fn-\fn^2$, then by taking the prime field $k_0 \subset k$, we see that $A$ will be $R_0$-unramified over the subring $R_0 = k_0[\pi]_{(\pi)} \subset A$.

\begin{prop}\label{struct_blowup}Let $(R_0,\pi)$ be a DVR with perfect residue field and suppose that $(A,\fn)$ is an $R_0$-unramified regular local ring. Let $X = \Spec A$, $\widetilde{X}$ the blowup along $\fn$, and $V(\opi)$ the closed subscheme of the exceptional divisor $E$ defined via $\gr_{\fn} A \twoheadrightarrow \gr_{\fn} A/\opi \gr_{\fn} A$. Let $x \in \widetilde{X}$. 
\begin{itemize}
\item[(a)] If $x \notin V(\opi)$, then $\widehat{\mco}_{\widetilde{X},x}$ is a power series ring over a field or DVR.
\item[(b)] If $x \in V(\opi)$, then there exists a complete DVR $(S,\pi)$ for which $\widehat{\mco}_{\widetilde{X},x}$ is a power series over $S[[U,V]]/(UV-\pi)$.
\end{itemize}
\end{prop}
\begin{proof}(a) If $x \notin V(\opi)$, Corollary \ref{ramification_locus} affirms that $\pi$ lies outside of the square of the maximal ideal of $\mco_{\widetilde{X},x}$. If $\pi$ is a unit, then $\mco_{\widetilde{X},x}$ contains the fraction field of $R_0$. Otherwise, $\mco_{\widetilde{X},x}$ is $R_0$-unramified. In either case the result follows by appealing to Lemma \ref{gen_cohen}.

(b) By part (a) of Lemma \ref{gen_cohen}, there is a complete DVR $(S,\pi)$ whose residue field coincides with that of $\mco_{\widetilde{X},x}$ and is equipped with an $R_0$-algebra morphism $S \to \widehat{\mco}_{\widetilde{X},x}$. From Lemma \ref{local_locus}, the kernel of $\mco_{\widetilde{X},x} \twoheadrightarrow \mco_{V(\opi),x}$ is generated by $(g,\pi/g)$ where $g$ is the local equation for the exceptional divisor. We now define $S[[U,V]]/(UV - \pi) \to \widehat{\mco}_{\widetilde{X},x}$ by mapping $U \mapsto g$ and $V \mapsto \pi/g$. Since $\widehat{\mco}_{\widetilde{X},x}/(U,V)\widehat{\mco}_{\widetilde{X},x} \cong \widehat{\mco}_{V(\opi),x}$ is regular, we can now apply part (b) of Lemma \ref{gen_cohen} to deduce the result.
  \end{proof}

\begin{mcor}\label{pos_blowup}Let $(R_0,\pi)$ be a DVR with perfect residue field and let $(A,\fn)$ be an $R_0$-unramified regular local ring. Put $X = \Spec A$ and let $\widetilde{X}$ be the blowup along $\fn$. Then Serre's Positivity Conjecture holds at all stalks $\mco_{\widetilde{X},x}$.
\end{mcor}
\begin{proof}By Proposition \ref{struct_blowup}, $\widehat{\mco}_{\widetilde{X},x}$ is a power-series over a zero, one, or two-dimensional base. The result now follows immediately from Theorem \ref{main}.
  \end{proof}

\subsection{Transversality in the Unramified Case}
We now fix $(R_0,\pi)$ a DVR with perfect residue field and an $R_0$-unramified regular local ring $(A,\fn)$. We let $X = \Spec A$ suppose that $Y$ and $Z$ are two closed integral subschemes meeting (set-theoretically) at a point and whose dimensions add up to that of $X$. In this situation, we are guaranteed that $\chi^{\mco_X}(\mco_Y,\mco_Z) \geq e(\mco_Y)e(\mco_Z)$ (see \cite[Thm. A]{Skalit}). Conjecture \ref{transversal} asserts that equality holds if and only if the tangent cones of $Y$ and $Z$ meet at a point --- that is, $\dim(\gr \mco_Y \otimes_{\gr \mco_X} \gr \mco_Z) = 0$. Using Rees' Theorem to connect the notion of integral closure of ideals to Hilbert-Samuel multiplicity theory, we have been able to show that Conjecture \ref{transversal} has an affirmative answer in many cases:

\begin{thm}\cite[Thm. B,C,D,E]{Skalit}\label{old}With $X$, $Y$, and $Z$ as above, assume that $\chi^{\mco_X}(\mco_Y,\mco_Z) = e(\mco_Y)e(\mco_Z)$.
\begin{itemize}
\item[(a)] If $A$ contains a field, then $\dim(\gr \mco_Y \otimes_{\gr \mco_X} \gr \mco_Z)=0$.
\item[(b)] If $\pi$ annihilates either $\mco_Y$ or $\mco_Z$, then $\dim(\gr \mco_Y \otimes_{\gr \mco_X} \gr \mco_Y)=0$.
\item[(c)] If $\pi$ annihilates neither $\mco_Y$ nor $\mco_Z$, then either $e(\mco_Y/\pi \mco_Y) = e(\mco_Y)$ or $e(\mco_Z/\pi \mco_Z) = e(\mco_Z)$. Furthermore, $\dim(\gr \mco_Y \otimes_{\gr \mco_X} \mco_Z) =0$ if any one of the following conditions holds:
		\begin{itemize}
		\item[(i) ] $e(\mco_Y/\pi \mco_Y) = e(\mco_Y)$ and $\gr \mco_Y$ has no height-one embedded primes.
		\item[(ii)] $e(\mco_Y/\pi \mco_Y) = e(\mco_Y)$ and $e(\mco_Z/\pi \mco_Z) = e(\mco_Z)$
		\item[(iii)] Either $Y$ or $Z$ has codimension $1$.
		\item[(iv)] $\gr \mco_Y \otimes_{\gr \mco_X} \gr \mco_Z \otimes_{\gr \mco_X} \frac{\gr \mco_X}{\opi \gr \mco_X}$ is zero-dimensional.
		\end{itemize}
\end{itemize}
\end{thm}

We now recast this transversality statement in the language of blowing up. For a more detailed discussion, consult Section 3.2 of \cite{Skalit}. Let $\widetilde{X}$, $\widetilde{Y}$, and $\widetilde{Z}$ be the point-blowups (i.e. along $\fn$) of $X$, $Y$, and $Z$. We can realize $\widetilde{Y}$ and $\widetilde{Z}$ as subschemes of $\widetilde{X}$ via the strict transforms of $Y$ and $Z$ under $\phi:\widetilde{X} \to X$. As reduced schemes,
\[ (\widetilde{Y} \cap \widetilde{Z})_{\red} \subset \phi^{-1}((Y \cap Z)_{\red}) = \phi^{-1}(\left\{\fn\right\}) = E = \Proj(\gr \mco_X) = \mathbb{P}_k^N \]
where $N = \dim X - 1$ and $k = A/\fn$. Hence as sets, $\widetilde{Y} \cap \widetilde{Z}$ coincides with the intersection of the projectivized tangent cones $E \cap \widetilde{Y} = \Proj(\gr \mco_Y)$ and \newline$E \cap \widetilde{Z} = \Proj(\mco_Z)$. To say that the tangent cones of $Y$ and $Z$ intersect at a point amounts to saying that $\widetilde{Y} \cap \widetilde{Z} = \emptyset$. Condition (iv) in Theorem \ref{old} says that if $\chi^{\mco_X}(\mco_Y,\mco_Z) = e(\mco_Y)e(\mco_Z)$, then $\widetilde{Y} \cap \widetilde{Z}$ is guaranteed to be empty as soon as one can show that $\widetilde{Y} \cap \widetilde{Z}$ misses the hyperplane $V(\opi)$ of $E$ which defines the ramification locus of $\pi$ (cf. Lemma \ref{ramification_locus}). However, with Theorem \ref{main} in hand, we have the following improvement, which is similar in flavor to \cite[Main Theorem (iii)]{Dutta_Blowup} and \cite{dutta_special_case_ii}.

\begin{mcor}Let $(R_0,\pi)$ be DVR with perfect residue field and let $(A,\fn)$ be an $R_0$-unramified regular local ring. Let $X = \Spec A$ and let $Y$ and $Z$ be closed, integral subschemes such that $\ell(\mco_Y \otimes_{\mco_X} \mco_Z) < \infty$ and $\dim Y + \dim Z = \dim X$. Suppose that $\chi^{\mco_X}(\mco_Y,\mco_Z) = e(\mco_Y)e(\mco_Z)$. If $\dim(\gr \mco_Y \otimes_{\gr \mco_X} \gr \mco_Z) \leq 1$, then, in fact, $\dim(\gr \mco_Y \otimes_{\gr \mco_X} \gr \mco_Z) = 0$.
\end{mcor}
\begin{proof}Let $\widetilde{X}$, $\widetilde{Y}$ and $\widetilde{Z}$ be the blowups at the closed point $\fn$. From the above discussion, the condition that $\dim(\gr \mco_Y \otimes_{\gr \mco_X} \gr \mco_Z) \leq 1$ means that $\widetilde{Y} \cap \widetilde{Z}$ is a finite (possibly empty) set of closed points $x_1, \cdots x_p \in \widetilde{X}$. In this case, the sheaves $\Tor_j^{\mco_{\widetilde{X}}}(\mco_{\widetilde{Y}},\mco_{\widetilde{Z}})$ are direct sums of skyscraper sheaves and therefore have no higher Zariski sheaf cohomology. The formula for $\chi^{\mco_{\widetilde{X}}}$ therefore reduces to
\[ \chi^{\mco_{\widetilde{X}}}(\mco_{\widetilde{Y}},\mco_{\widetilde{Z}}) = \sum_{i,j \geq 0}{(-1)^{i+j}\ell(H^i(\widetilde{X}, \Tor_j^{\mco_{\widetilde{X}}}(\mco_{\widetilde{Y}},\mco_{\widetilde{Z}})))} = \sum_{i = 1}^p{\chi^{\mco_{\widetilde{X},x_i}}(\mco_{\widetilde{Y},x_i},\mco_{\widetilde{Z},x_i})} \]
Note that if $x \in \widetilde{X}$ is closed, then by Lemma 3.8 of \cite{Skalit}, there are equalities
\[ \dim(\mco_{\widetilde{X},x}) = \dim X, \hspace{5mm} \dim(\mco_{\widetilde{Y},x}) = \dim Y, \hspace{5mm}\mbox{and}\hspace{5mm} \dim(\mco_{\widetilde{Z},x}) = \dim Z. \]
By Corollary \ref{pos_blowup}, Serre's Positivity Conjecture is known for every local ring of $\widetilde{X}$, so each term $\chi^{\mco_{\widetilde{X},x_i}}(\mco_{\widetilde{Y},x_i},\mco_{\widetilde{Z},x_i})$ is positive. We therefore see that $\chi^{\mco_{\widetilde{X}}}(\mco_{\widetilde{Y}},\mco_{\widetilde{Z}}) = 0$ if and only if $\widetilde{Y} \cap \widetilde{Z} = \emptyset$. On the other hand, since we have assumed that $\chi^{\mco_X}(\mco_Y,\mco_Z) = e(\mco_Y)e(\mco_Z)$, Fulton's formula \cite[20.4.3]{Fulton}
\[ \chi^{\mco_{X}}(\mco_Y,\mco_Z) = e(\mco_Y)e(\mco_Z) + \chi^{\mco_{\widetilde{X}}}(\mco_{\widetilde{Y}},\mco_{\widetilde{Z}}) \]
forces $\chi^{\mco_{\widetilde{X}}}(\mco_{\widetilde{Y}},\mco_{\widetilde{Z}}) = 0$.
  \end{proof}

\section*{Acknowledgments}
The author wishes to thank M. V. Nori for many helpful discussions, particularly during the early phases of this project. Additional thanks are owed to the anonymous referee for thoroughly reading the manuscript, catching various misprints, and offering a number of helpful suggestions to improve the overall readability. This work was supported in part by an NSF Research and Training Grant (DMS-1246844).

\bibliographystyle{alpha}
\bibliography{unramified_paper}

\end{document}